\documentclass[11pt]{amsart}

\usepackage{amsmath,amsthm,amsfonts,amssymb,bm,graphicx }

\usepackage{graphicx}
\usepackage{enumitem}
\usepackage{xparse}
\usepackage{tikz}
\usepackage{comment}
\usepackage{bm}

\usepackage
{hyperref}
\hypersetup{colorlinks=true,citecolor=blue,linkcolor=blue,urlcolor=blue}

\makeatletter
\g@addto@macro\bfseries{\boldmath}
\makeatother

\theoremstyle{plain}
\newtheorem{theorem}{Theorem}
\newtheorem*{theorem*}{Theorem}

\newtheorem{prop}[theorem]{Proposition}

\theoremstyle{remark}

\numberwithin{theorem}{section}

\numberwithin{equation}{section}

\def\N{\mathbb N}
\def\Z{\mathbb Z}
\def\R{\mathbb R}
\def\Q{\mathbb Q}
\def\C{\mathbb C}

\def\O{\mathcal O}
\def\P{\mathcal P}

\begin{document}


\author{Magdal\'ena Tinkov\'a}

\title{Bounds on the Pythagoras number and indecomposables in biquadratic fields}

\address{Department of Algebra, Faculty of Mathematics and Physics, Charles University, 
Sokolovsk\'{a} 83, 18600 Praha 8, Czech Republic}
\address{Faculty of Information Technology, Czech Technical University in Prague, Th\'akurova 9, 160 00 Praha 6, Czech Republic}
\email{tinkova.magdalena@gmail.com}
\address{TU Graz, Institute of Analysis and Number Theory, Kopernikusgasse 24/II, 8010 Graz, Austria}



\thanks{The author was supported by Czech Science Foundation GA\v{C}R, grants 21-00420M and 22-11563O.}

\begin{abstract} 
We show that for all real biquadratic fields not containing $\sqrt{2}$, $\sqrt{3}$, $\sqrt{5}$, $\sqrt{6}$, $\sqrt{7},$ and $\sqrt{13}$, the Pythagoras number of the ring of algebraic integers is at least $6$. We will also provide an upper bound on the norm and the minimal (codifferent) trace of additively indecomposable integers in some families of these fields.     
\end{abstract}

\setcounter{tocdepth}{2}  \maketitle 

\section{Introduction}

In the first part of this paper, we focus on the study of the so-called Pythagoras number in real biquadratic fields. For that, let us consider some commutative ring $\O$ and its subset $\sum \O^2$, which contains those $\alpha\in \O$ which can be written as the sum of squares of elements in $\O$. Then the Pythagoras number $\P(\O)$ of $\O$ is the minimal number in $\N\cup\{\infty\}$ such that every $\alpha\in\sum \O^2$ can be written as the sum of at most $\P(\O)$ squares of elements in $\O$.

The exceptional case when $\O$ is a field is well-known and richly studied. However, in this paper, we will discuss the case when $\O\subseteq \O_K$ is an order in a totally real number field $K$, and $\O_K$ is the ring of algebraic integers of $K$. In comparison with results for fields, we do not have much information about such orders.

In that case, $\P(\O)<\infty$. However, it can also attain arbitrarily large values by the result of Scharlau \cite{Sc}. In his proof, Scharlau efficiently uses properties of multiquadratic number fields to get an increasing lower bound for some particular family of these fields. Regarding the upper bound, we have $\P(\O)\leq f(d)$ where the function $f$ depends only on the degree $d$ of the field $K$ (see, e.g., \cite[Corollary 3.3]{KY1}). Furthermore, one can take $\P(\O)\leq d+3$ whenever $2\leq d\leq 5$.

Except for these general statements, we only have some results for several families of number fields. Peters \cite{Pet} investigated the case of orders $\O\subseteq\O_K$ where $K$ is a real quadratic field. He proved that $\P(\O)=5$ with the exception of the following orders: $\P(\Z[\sqrt{5}])=\P(\Z[\sqrt{6}])=\P(\Z[\sqrt{7}])=4$ and $\P(\Z[\sqrt{2}])=\P(\Z[\sqrt{3}])=\P\big(\Z[\frac{1+\sqrt{5}}{2}]\big)=3$, which were mostly resolved before. Thus, in quadratic fields, the upper bound $5$ from the above result is attained in all but finitely many cases.

The case of the simplest cubic fields, which are generated by a root $\rho$ of the polynomial $x^3-ax^2-(a+3)x-1$ with $a\in\Z$, $a\geq-1$, was discussed in \cite{Ti2}. In that paper, the present author proves that $\P(\Z[\rho])=6$ if $a\geq 3$, for which she used properties and knowledge of additively indecomposable integers in these fields. The fact that $P(\Z[\rho])=P(\O_K)=4$ for $a=-1$ was consequently shown in \cite{Kr}. The method developed in \cite{Ti2} was also applied to other cases of totally real cubic fields \cite{GMT,KST}. 

For biquadratic fields, the above upper bound is $7$. In \cite{KRS}, the authors proved that for an infinite subfamily of biquadratic fields $K$, we indeed have $\P(\O_K)=7$. On the other hand, they also discuss one concrete infinite subfamily of biquadratic fields (containing $\sqrt{5}$), for which $P(\O_K)=5$. Note that this was later extended to quartic fields containing $\sqrt{2}$ in \cite{HH}. Moreover, in \cite{KRS}, one can find more results regarding the upper and lower bounds on $\P(\O_K)$, and we will refine some of them. In particular, their Theorems 8.1 and Proposition 8.3 say that if $n\neq 2,3,5,6,7$ is a square-free positive integer, then there can exist only finitely many real biquadratic fields $K$ such that $\sqrt{n}\in K$ and $\P(\O_K)<6$. Under a similar restriction, we will show that $\P(\O_K)\geq 6$ is always satisfied by proving the following theorem: 

\begin{theorem} \label{thm:main}
Let $K$ be a real biquadratic field not containing $\sqrt{2}$, $\sqrt{3}$, $\sqrt{5}$, $\sqrt{6}$, $\sqrt{7},$ and $\sqrt{13}$. Then $\P(\O_K)\geq 6$.
\end{theorem}

Note that in \cite[Proof of Theorem 8.1]{KRS}, authors need to discuss $12$ different cases of biquadratic fields. In the proof of Theorem \ref{thm:main}, we will distinguish much less cases by using some properties related to their integral bases. 

Moreover, we will prove that the following subfamily of biquadratic fields has the maximal possible Pythagoras number for degree $4$:

\begin{theorem} \label{thm:main7}
Let $n\geq 6$ be a rational integer such that $p=(2n-1)(2n+1)$, $q=(2n-1)(2n+3)$ and $r=(2+1)(2n+3)$ are square-free integers, and let $K=\Q(\sqrt{p},\sqrt{q})$. Then $\P(\O_K)=7$.
\end{theorem}
As mentioned before, the authors in \cite{KRS} also show one infinity family of biquadratic fields with Pythagoras number $7$. However, our family has a different type of integral basis (see Section \ref{sec:preli}), and, moreover, in our case, $\gcd(p,q)\neq 1$. Therefore, it further confirms that $\P(\O_K)=7$ is not so rare for biquadratic fields.   

\bigskip

The second part of this paper is devoted to a deeper look at additively indecomposable integers in real biquadratic fields. Let $\O_K^{+}\subset\O_K$ be the set of totally positive algebraic integers in $\O_K$. Recall that $\alpha\in K$ is totally positive if all its Galois conjugates are positive. Then $\alpha\in\O_K^{+}$ is indecomposable in $\O_K$ if $\alpha$ cannot be written as a sum of totally positive algebraic integers.

Especially in the last few years, the indecomposable integers were heavily used in the study of universal quadratic forms over number fields \cite{BK, Ka, Ya, KS, Ka3, KT}, which is one of the classical topics in the number theory. For more details in this direction, see also the survey of Kala \cite{Ka4}. Moreover, as indicated before, these elements have a nice application in the determination of the Pythagoras number \cite{Ti2, GMT, KST}.

On the other hand, we have more information about these elements only for degrees $2$, $3$ and $4$. Let us describe these results in detail. The structure of indecomposable integers in real quadratic fields can be determined from the continued fraction of concrete quadratic irrationalities as proved by Perron \cite{Pe} and Dress and Scharlau \cite{DS}. Note that at the moment, we do not know a similar relation between indecomposable integers and, e.g., multidimensional continued fractions, although in \cite{KST}, we started some investigation in this direction. Then, the structure of indecomposable integers was determined for some families of monogenic \cite{KT, Ti1} and non-mogenic \cite{GMT} cubic fields. 

Of course, the most relevant case for this paper is the case of real biquadratic fields. The research in this direction started with \cite{CLSTZ}, where we were trying to answer the following question: Are all indecomposable integers from quadratic fields indecomposable in the larger biquadratic fields? In that paper, we give some sufficient conditions when this is true. However, in \cite{KTZ}, we provide a counterexample. This problem was recently almost solved by Man \cite{Man}, who proved that indecomposable integers from at least two quadratic subfields cannot decompose. Moreover, Man was able to use techniques developed in \cite{KT} to find the structure of indecomposables for some subfamilies of biquadratic fields. In the following, we will benefit from this result.

What is known about indecomposable integers in general is that their algebraic norm is bounded by some constant depending only on $K$ \cite{Bru}. This result was later refined by Kala and Yatsyna \cite{KY2} by proving that we can take the discriminant of $K$ as this upper bound. However, even that is not always the best possible. For quadratic fields, where we can use properties of the continued fractions, there appeared a whole series of articles improving this bound \cite{DS, JK, Ka2, TV}. Similar results for several families of cubic fields followed \cite{Ti1, GMT}. In this spirit, we prove the following theorem:

\begin{theorem} \label{thm:main_norms}
Let $n\geq 6$ be a rational integer such that $p=(2n-1)(2n+1)$, $q=(2n-1)(2n+3)$ and $r=(2+1)(2n+3)$ are square-free integers.
If $\alpha\in\O_K^{+}$ is indecomposable in $\O_K$ for $K=\Q(\sqrt{p},\sqrt{q})$, then $N(\alpha)\leq 16n^4+64n^3+16n^2-96n+36$.
\end{theorem}

Since the discriminant of these fields is $16pqr\sim 1024n^6$, there is a great discrepancy between the above bound and the bound derived by Kala and Yatsyna. We provide a similar result for two more families of biquadratic fields. 

Moreover, we will also study the so-called minimal (codifferent) traces of indecomposables $\alpha$. By that, we mean the minimum of traces of $\alpha\delta$ where $\delta$ runs over all totally positive elements of the codifferent $\O_K^{\vee}$ (for more details, see Section \ref{sec:preli}). If $\min_{\delta\in\O_K^{\vee,+}}\left(\alpha\right)=1$ for $\alpha\in\O_K^{+}$, then, easily, $\alpha$ is indecomposable in $\O_K$. The opposite implication is not true in general, although it is valid for real quadratic fields \cite{KT}. 
On the other hand, we know families of cubic fields with this minimal trace equal to $1$ or $2$ \cite{KT}, or a family of cubic orders, for which it can attain arbitrarily large values \cite{Ti2}. In this paper, we prove the following:

\begin{theorem} \label{thm:main_traces}
Let $n\geq 6$ be a rational integer such that $p=(2n-1)(2n+1)$, $q=(2n-1)(2n+3)$ and $r=(2+1)(2n+3)$ are square-free integers.
If $\alpha\in\O_K^{+}$ is indecomposable in $\O_K$ for $K=\Q(\sqrt{p},\sqrt{q})$, then $\min_{\delta\in\O_K^{\vee,+}}\left(\alpha\right)\leq 2$.
\end{theorem} 

Therefore, this family is an example of biquadratic fields for which these minimal traces are small. Moreover, we show subfamilies of such fields for which this minimal trace is $1$ for all indecomposable integers; such examples were known before only for degree $2$. Note that these kinds of results can serve as an alternative proof of the indecomposability of these elements. Using this, we will also be able to find a lower bound on the number of variables of universal quadratic forms. Note that indecomposable integers of small codifferent traces can be also used to resolve the existence of $\Z$-forms over $\O_K$ -- see, e.g., \cite{KL} for biquadratic case.  

In Section \ref{sec:preli}, we provide the necessary facts, notation, and tools we use in the paper. Section \ref{sec:pyth} is devoted to the proof of Theorem \ref{thm:main}, for which we discuss several different cases separately, and we will also prove Theorem \ref{thm:main7}. Finally, in Section \ref{sec:indebiqua}, we give some new results for three families of biquadratic fields, including proofs of Theorems \ref{thm:main_norms} and \ref{thm:main_traces}.

\section{Preliminaries} \label{sec:preli}

Let $K$ be a totally real number field and denote by $\O_K$ its ring of algebraic integers. We say $\alpha\in\O_K$ is totally positive if $\sigma(\alpha)>0$ for every embedding $\sigma$ of $K$ into complex numbers. We denote the set of such algebraic integers by $\O_K^{+}$. An element $\alpha\in\O_K^{+}$ is indecomposable if $\alpha\neq \beta+\gamma$ for any $\beta,\gamma\in\O_K^{+}$. We define the trace and norm of $\alpha\in K$ as $\text{Tr}(\alpha)=\sum_{\sigma}\sigma(\alpha)$ and $\prod_{\sigma}\sigma(\alpha)$, respectively, where $\sigma$ runs over all embedding of $K$ into $\C$. Moreover, we say that elements $\alpha,\beta\in\O_K$ are associated if there is a unit $\varepsilon\in\O_K$ such that $\alpha=\beta\varepsilon$.       

By the codifferent of $K$, we understand the set of the form
\[
\O_K^{\vee}=\{\delta\in K; \text{Tr}(\alpha\delta)\in\Z\text{ for all }\alpha\in\O\}.
\] 
If $\O_K=\Z[\gamma]$ for some algebraic integer $\gamma$ with the minimal polynomial $f$, then $\O_K^{\vee}=\frac{Z[\gamma]}{f'(\gamma)}$ where $f'$ is the derivative of $f$. However, in the other cases, we are able to determine the form of the codifferent using \cite[Proposition 4.14]{Na}. Let ${\gamma_1,\ldots,\gamma_d}$ be an integral basis of $\O_K$. Then if $\varphi_1,\ldots,\varphi_d\in K$  satisfy $\text{Tr}(\gamma_i\varphi_j)=\delta_{ij}$ where $\delta_{ij}$ is the Kronecker symbol, then ${\varphi_1,\ldots,\varphi_d}$ is a $\Z$-basis of $\O_K^{\vee}$. Using this, we can easily compute the form of ${\varphi_1,\ldots,\varphi_d}$, and, moreover, we have
\[
\text{Tr}\left(\Bigg(\sum_{i=1}^d a_i\gamma_i\Bigg)\Bigg(\sum_{j=1}^d b_j\varphi_j\Bigg)\right)=\sum_{i=1}^d a_ib_i
\]
for $a_i,b_j\in\Z$. We will denote $\text{minTr}(\alpha)=\min_{\delta\in\O_K^{\vee,+}}\text{Tr}(\alpha\delta)$ for $\alpha\in\O_K^{+}$ and call it the minimal (codifferent) trace of $\alpha$.

By a quadratic form over $\O_K$ in $n$ variables, we mean $Q(x_1,\ldots,x_n)=\sum_{1\leq i,j\leq n}a_{ij}x_ix_j$ where $a_{ij}\in\O_K$. The form $Q$ is totally positive definite if $Q(\alpha_1,\ldots,\alpha_n)\in\O_K^{+}$ for all $\alpha_i\in\O_K$, $(\alpha_1,\ldots,\alpha_n)\neq (0,\ldots,0)$. In this paper, we will always assume that our quadratic forms are totally positive. We say that $Q$ is classical if $2|a_{ij}$ for all $i\neq j$, and diagonal if $a_{ij}=0$ for all $i\neq j$. Moreover, $Q$ is universal if every element in $\O_K^{+}$ can be represented by $Q$. 

In the following we will benefit from the following result:

\begin{prop}[{\cite[Section 7]{KT}}] \label{prop:unilowvar}
Let $K$ be a totally real number field of degree $4$. Assume that there exist $\alpha_1,\ldots, \alpha_n\in\O_K^{+}$ such that $\textup{Tr}(\alpha_i\delta_1)=1$ for $\delta_1\in\O_K^{\vee,+}$ and $1\leq i\leq n$. Moreover, assume that there exist $\beta_1,\ldots, \beta_m\in\O_K^{+}$ such that $\textup{Tr}(\beta_i\delta_2)=2$ for $\delta_2\in\O_K^{\vee,+}$ and $1\leq i\leq m$. Then every classical universal quadratic form over $\O_K$ has at least $\frac{n}{4}$ variables, and every diagonal universal quadratic form over $\O_K$ has at least $\max\left(\frac{n}{4},\frac{m}{12}\right)$ variables.    
\end{prop}  

Let us now introduce the Pythagoras number of a ring. Let $\O$ be a commutative ring. Let us denote by $\sum \O^2$ the set of $\alpha\in\O$ such that $\alpha$ can be expressed as the sum of elements in $\O$, and let $\sum^m \O^2\subseteq\sum \O^2$ be the subset of those elements, for which we need at most $m$ squares. Then, by the Pythagoras number of $\O$, we mean
\[
\P(\O)=\inf\Big\{m\in\N\cup\{+\infty\};\; \sum \O^2=\sum^m \O^2\Big\}.
\]
Moreover, we will focus on the case when $\O = \O_K$ and $K$ is a real biquadratic field. In that case, we know that $\P(\O_K)\leq 7$, see, e.g., \cite[Corollary 3.3]{KY1}. 

In this paper, we study real biquadratic fields $K=\Q(\sqrt{p},\sqrt{q})$, where $p,q>1$ are distinct square-free integers. We will denote by $r$ the remaining square root contained in $K$, i.e., $r=\frac{pq}{\gcd(p,q)^2}$. Let $\alpha=x+y\sqrt{p}+z\sqrt{q}+w\sqrt{r}\in\Q(\sqrt{p},\sqrt{q})$. The embeddings of $\Q(\sqrt{p},\sqrt{q})$ into $\R$ are given as
\begin{align*}
\sigma_1(\alpha)&=x+y\sqrt{p}+z\sqrt{q}+w\sqrt{r},\\
\sigma_2(\alpha)&=x-y\sqrt{p}+z\sqrt{q}-w\sqrt{r},\\
\sigma_3(\alpha)&=x+y\sqrt{p}-z\sqrt{q}-w\sqrt{r},\\
\sigma_4(\alpha)&=x-y\sqrt{p}-z\sqrt{q}+w\sqrt{r}.
\end{align*}
Depending on the values of $p$ and $q$ (after possible interchanging of $p$, $q$ and $r$), we can distinguish the following cases of integral bases of $K$ \cite{Ja,Wi}:
\begin{enumerate}
\item if $p\equiv 2\;(\text{mod }4)$ and $q\equiv 3\;(\text{mod }4)$, then $\O_K=\Z\Big[1,\sqrt{p},\sqrt{q},\frac{\sqrt{p}+\sqrt{r}}{2}\Big]$,
\item if $p\equiv 2\;(\text{mod }4)$ and $q\equiv 1\;(\text{mod }4)$, then $\O_K=\Z\Big[1,\sqrt{p},\frac{1+\sqrt{q}}{2},\frac{\sqrt{p}+\sqrt{r}}{2}\Big]$,
\item if $p\equiv 3\;(\text{mod }4)$ and $q\equiv 1\;(\text{mod }4)$, then $\O_K=\Z\Big[1,\sqrt{p},\frac{1+\sqrt{q}}{2},\frac{\sqrt{p}+\sqrt{r}}{2}\Big]$,
\item if $p,q\equiv 1\;(\text{mod }4)$, then
\begin{enumerate}
\item $\O_K=\Z\Big[1,\frac{1+\sqrt{p}}{2},\frac{1+\sqrt{q}}{2},\frac{1+\sqrt{p}+\sqrt{p}+\sqrt{r}}{4}\Big]$ if $\frac{p}{\gcd(p,q)}\equiv \frac{q}{\gcd(p,q)}\equiv 1\;(\text{mod }4)$,
\item $\O_K=\Z\Big[1,\frac{1+\sqrt{p}}{2},\frac{1+\sqrt{q}}{2},\frac{1-\sqrt{p}+\sqrt{p}+\sqrt{r}}{4}\Big]$ if $\frac{p}{\gcd(p,q)}\equiv \frac{q}{\gcd(p,q)}\equiv 3\;(\text{mod }4)$.
\end{enumerate} 
\end{enumerate}
Moreover, if we set $p_0=\gcd(q,r)$, $q_0=\gcd(p,r)$ and $r_0=\gcd(p,q)$, we can write $p$, $q$ and $r$ as $p=q_0r_0$, $q=p_0r_0$ and $r=p_0q_0$.

Note that in totally real fields of degree $4$, an element $\alpha\in\O_K$ is totally positive if and only if it is a root of a polynomial of the form
\[
x^4-Ax^3+Bx^2-Cx+D
\]
where $A,B,C,D\in\Z$, $A,B,C,D>0$. In biquadratic fields, it is an easy computational task to find coefficients of the polynomial above, and for that, we can use Vieta's formulas.

\section{Pythagoras number in biquadratic fields} \label{sec:pyth}

In this part, our goal is to prove Theorem \ref{thm:main}, and for that, we will solve several different cases separately. 
First of all, we will focus on integral basis of type (1) as introduced above in Section \ref{sec:preli}.

\begin{prop} \label{prop:B1}
Let $K=\Q(\sqrt{p},\sqrt{q})$ where $p,q>1$ are square-free positive integers such that $p\equiv 2\;(\textup{mod }4)$, $q\equiv 3\;(\textup{mod }4)$ and $p,q,r\geq 10$. Then $\P(\O_K)\geq 6$. 
\end{prop}  

\begin{proof}
Without loss of generality, we can assume $p<r$. First of all, suppose $q<r$, i.e., we get $r_0<p_0,q_0$. To obtain the lower bound, let us consider the element of $\O_K$ of the form
\[
\alpha = 7+(1+\sqrt{p})^2+(1+\sqrt{q})^2=9+p+q+2\sqrt{p}+2\sqrt{q}.
\]
In the following, we will derive the minimal number of squares that we need to express $\alpha$. Let
\[
\alpha=\sum_{i=1}^N \Big(a_i+\frac{b_i}{2}\sqrt{p}+c_i\sqrt{q}+\frac{d_i}{2}\frac{\sqrt{pq}}{r_0}\Big)^2
\]
where $a_i,b_i,c_i,d_i\in\Z$ and $b_i\equiv d_i\;(\textup{mod }2)$ for all $1\leq i\leq N$, and $N\in\N$. Thus, we get the following equalities:
\begin{align}
9+p+q&=\sum_{i=1}^N\Big(a_i^2+\frac{b_i^2p}{4}+c_i^2q+\frac{d_i^2pq}{4r_0^2}\Big),\label{eq:1coef}\\
2&=\sum_{i=1}^N( a_ib_i+c_id_ip_0),\label{eq:2coef}\\
2&=\sum_{i=1}^N\Big( 2a_ic_i+\frac{1}{2}b_id_iq_0\Big),\label{eq:3coef}\\
0&=\sum_{i=1}^N( b_ic_ir_0+a_id_i)\label{eq:4coef}.
\end{align}

If all the coefficients $b_i$ were equal to $0$, the equality (\ref{eq:2coef}) would imply $2=\sum_{i=1}^N c_id_ip_0$, which is not possible as $p_0\geq 3$ (since $r_0<p_0,q_0$ and $p_0$ is odd). Thus, we have at least one non-zero coefficient $b_i$. Using (\ref{eq:2coef}), we can make the same conclusion for coefficients $a_i$. On the other hand, if all the coefficients $c_i$ were zeros, the equality (\ref{eq:3coef}) would give $2=\frac{q_0}{2}\sum_{i=1}^N b_id_i$, which is again not possible since $p=q_0r_0\geq 10$ and $r_0<q_0$, implying $\frac{q_0}{2}\geq 3$.

Let us now focus on the case when $d_j\neq 0$ for some $j\in\{1,\ldots,N\}$. If $d_j$ were even, the right side of (\ref{eq:1coef}) would be at least $1+\frac{p}{4}+q+\frac{pq}{r_0}$, and the necessary inequality
\[
1+\frac{p}{4}+q+\frac{pq}{r_0^2}\leq 9+p+q
\]
is equivalent to $q_0(4p_0-3r_0)\leq 32$. This cannot be satisfied since $q_0\geq 6$ and $4p_0-3r_0\geq 4(r_0+2)-3r_0=r_0+8\geq 9$. On the other hand, if $d_j$ were odd, (\ref{eq:3coef}) would imply the existence of $k\neq j$ such that $d_k\neq 0$, and in addition, $b_k\neq 0$ (note that in this case, clearly, $b_j\neq 0$). Therefore, as before, we get
\[
1+\frac{p}{2}+q+\frac{pq}{2r_0^2}\leq 9+p+q,
\]
which can be rewritten as $q_0(p_0-r_0)\leq 16$. Since $q_0\geq 6$ and $p_0-r_0\geq 2$, our inequality can be fulfilled only if $q_0=6$ and $p_0=r_0+2$. However, the only possible case satisfying these conditions is when $p=30$ and $q=35$. Using a computer program, we can easily check that $\P(\O_K)\geq 6$ even in this particular case (we can consider, e.g., the element $44+\sqrt{30}+\sqrt{42}$).

Thus, $d_i=0$ for all $1\leq i\leq N$. Since $b_i$ is congruent to $d_i$ modulo $2$, we obtain $b_i=2b'_i$. Moreover, if we had two nonzero coefficients $b'_i$, we would obtain $1+2p+q\leq 9+p+q$, which is impossible for $p\geq 10$. The same must be true for coefficients $c_i$. Thus, we have exactly one non-zero coefficient $b'_j$, and exactly one non-zero coefficient $c_k$, and, clearly, $|b'_j|=|c_k|=1$. 

Under this condition, $2=\sum_{i=1}^N 2a_ib'_i$ (i.e., (\ref{eq:2coef})) implies $|a_j|=1$ and $\text{sgn}(a_j)=\text{sgn}(b'_j)$. Moreover, from $2=\sum_{i=1}^N 2a_ic_i$ (i.e., (\ref{eq:3coef})), we have $|a_k|=1$ and $\text{sgn}(a_k)=\text{sgn}(c_k)$. From $0=\sum_{i=1}^N 2b'_ic_ir_0$ (i.e, (\ref{eq:4coef})), we obtain $c_j=b'_k=0$, i.e., $j\neq k$. Thus, without loss of generality, we can put $a_j=b'_j=1$ and $a_k=c_k=1$. It follows that every relevant decomposition of $\alpha$ must be of the form
\[
\alpha=7+(1+\sqrt{p})^2+(1+\sqrt{q})^2,
\]
and since for $7$, we need at least $4$ squares for $p,q,r\geq 10$, we indeed get $\P(\O_K)\geq 6$ in this case.

\bigskip

Now, let us focus on the case when $r<q$, i.e., $q_0<r_0<p_0$. Here, we will consider the element of the form
\[
\beta=7+\Bigg(1+\frac{\sqrt{p}+\frac{\sqrt{pq}}{r_0}}{2}\Bigg)^2+\Bigg(\frac{\sqrt{p}-\frac{\sqrt{pq}}{r_0}}{2}\Bigg)^2=8+\frac{p}{2}+\frac{pq}{2r_0^2}+\sqrt{p}+\sqrt{r}.
\]   
We get equalities of the form     
\begin{align}
8+\frac{p}{2}+\frac{pq}{2r_0^2}&=\sum_{i=1}^N\Big(a_i^2+\frac{b_i^2p}{4}+c_i^2q+\frac{d_i^2pq}{4r_0^2}\Big),\label{eq2:1coef}\\
1&=\sum_{i=1}^N( a_ib_i+c_id_ip_0),\label{eq2:2coef}\\
0&=\sum_{i=1}^N\Big( 2a_ic_i+\frac{1}{2}b_id_iq_0\Big),\label{eq2:3coef}\\
1&=\sum_{i=1}^N( b_ic_ir_0+a_id_i)\label{eq2:4coef}.
\end{align}
If all the coefficients $b_i$ (or $a_i$) were equal to zero, we would have $1=p_0\sum_{i=1}^N c_id_i$, which is impossible for $p_0\geq 3$. Similarly, all zero coefficients $d_i$ imply $1=r_0\sum_{i=1}^N b_ic_i$, which is in comtradiction with $r_0\geq 3$. Thus, we have at least one non-zero $a_i$, $b_i$, and $d_i$. 

Now, let us assume that at least one $c_i$ is non-zero. It implies that we must have
\[
1+\frac{p}{4}+q+\frac{pq}{4r_0^2}\leq 8+\frac{p}{2}+\frac{pq}{2r_0^2},
\]
which can be rewritten as $4q\leq 28+p+r$. However, $p,r<q$, and, moreover, $10\leq q\leq 14$ only if $q=11$. Nevertheless, for $q=11$, we cannot have $10\leq p<r<q=11$. Thus, it implies that all coefficients $c_i$ are zero.

If $|d_i|\geq 2$, we would get
\[
1+\frac{p}{4}+\frac{pq}{r_0^2}\leq 8+\frac{p}{2}+\frac{pq}{2r_0^2},
\]
which leads to $28\geq 2r-p = r-p+r$. That can be satisfied only for $10<r\leq 24$, which gives square-free $r$ only if $r=14,22$. However, only the pair $p=10$ and $r=14$ fulfills $ r-p+r\leq 28$, and for it, we can check by a computer program that for the element $\beta=20+\sqrt{10}+\sqrt{14}$, we need at least $6$ squares. Thus, $|d_i|\leq 1$ for all $1\leq i\leq N$.

The equality (\ref{eq:3coef}) gives $\sum_{i=1}^N b_id_i=0$. Since we have at least one $|d_i|=1$ with $b_i$ odd, there must exist $j\neq i$ such that $|d_j|=1$, and thus $b_j$ is also odd. Moreover, there exists an even number of such pairs. Assume that there are more than two such pairs. It implies
\[
1+p+r\leq 8+\frac{p}{2}+\frac{pq}{2r_0^2},
\] 
i.e., $p+r\leq 14$, which is impossible for $p,r\geq 10$. Thus, there are exactly two pairs of odd $b_i$ and $d_i$, lets say $(b_j,d_j)$ and $(b_k,d_k)$. Moreover, similarly as before, it can be easily shown that $|b_j|=|b_k|=1$, and $b_i=0$ for $i\neq j,k$.

Since $\sum_{i=1}^N b_id_i=0$, without loss of generality, we can suppose $b_j=1$, $d_j=1$, $b_k=1$ and $d_k=-1$. The equalities $1=\sum_{i=1}^Na_ib_i=a_j+a_k$ and $1=\sum_{i=1}^Na_id_i=a_j-a_k$ then give $a_j=1$ and $a_k=0$. Thus, every our decomposition of $\beta$ must be of the form 
\[
\beta=7+\Bigg(1+\frac{\sqrt{p}+\frac{\sqrt{pq}}{r_0}}{2}\Bigg)^2+\Bigg(\frac{\sqrt{p}-\frac{\sqrt{pq}}{r_0}}{2}\Bigg)^2,
\]
and for $7$, we need at least $4$ squares. Thus, $\P(\O_K)\geq 6$ even in this case.     
\end{proof}

We will proceed with bases of types (2) and (3).

\begin{prop} \label{prop:B23}
Let $K=\Q(\sqrt{p},\sqrt{q})$ where $p,q>1$ are square-free positive integers such that $p\equiv 2,3\;(\textup{mod }4)$, $q\equiv 1\;(\textup{mod }4)$, $p,r\geq 10$ and $q\geq 17$. Then $\P(\O_K)\geq 6$. 
\end{prop}  

\begin{proof}
As before, without loss of generality, we can suppose $p<r$; thus $r_0<p_0$. Moreover, since the case $r_0=1$ was already resolved in \cite[Proposition 5.4]{KRS}, we can also assume $r_0\geq 3$. 

Moreover, in the first part of this proof, we will suppose $q_0\geq 2$, i.e., that $p$ and $r$ are not coprime. 
Let us focus on the element
\[
\alpha = 7 + \Bigg(1+\frac{\sqrt{p}+\frac{\sqrt{pq}}{r_0}}{2}\Bigg)^2+\Bigg(\frac{1-\sqrt{q}}{2}\Bigg)^2=\frac{33}{4}+\frac{p}{4}+\frac{q}{4}+\frac{pq}{4r_0^2}+\sqrt{p}+\frac{q_0-1}{2}\sqrt{q}+\sqrt{r}.
\]
Here, we consider equalities of the form
\begin{align}
\frac{33}{4}+\frac{p}{4}+\frac{q}{4}+\frac{pq}{4r_0^2}&=\sum_{i=1}^N\Big(\frac{a_i^2}{4}+\frac{b_i^2p}{4}+\frac{c_i^2q}{4}+\frac{d_i^2pq}{4r_0^2}\Big),\label{eq23b:1coef}\\
1&=\sum_{i=1}^N\Big(\frac{1}{2} a_ib_i+\frac{1}{2}c_id_ip_0\Big),\label{eq23b:2coef}\\
\frac{q_0-1}{2}&=\sum_{i=1}^N\Big( \frac{1}{2}a_ic_i+\frac{1}{2}b_id_iq_0\Big),\label{eq23b:3coef}\\
1&=\sum_{i=1}^N\Big(\frac{1}{2} b_ic_ir_0+\frac{1}{2}a_id_i\Big)\label{eq23b:4coef}
\end{align}
where $a_i,b_i,c_i,d_i\in\Z$.
As in the previous case, since $r_0,p_0\geq 3$ and $q_0\geq 2$, we must have at least one non-zero coefficient $a_i$, $b_i$, $c_i$ and $d_i$. 

If this non-zero $d_i$ were even, the right side of (\ref{eq23b:1coef}) would be at least
\[
\frac{1}{4}+\frac{p}{4}+\frac{q}{4}+\frac{pq}{r_0^2},
\]
which is lower than the left side only if $3r\leq 32$. That is impossible for $r\geq 14$. On the other hand, if there existed even non-zero $b_i$, the corresponding inequality
\[
\frac{1}{4}+\frac{5p}{4}+\frac{q}{4}+\frac{pq}{4r_0^2}\leq \frac{33}{4}+\frac{p}{4}+\frac{q}{4}+\frac{pq}{4r_0^2} 
\] 
would be satisfied only when $p\leq 8$, which is excluded in the statement of this proposition. Thus, we have only non-zero odd pairs of coefficients $b_i$ and $d_i$. Furthermore, if there were more than one such pair, the inequality needed in (\ref{eq23b:1coef}) would lead to $p+r\leq 32$, which is not true for any case of type (2) or (3) if $r_0\geq 3$. Therefore, we have only one non-zero pair $b_j$ and $d_j$, and, clearly, $|b_j|=|d_j|=1$.

Similarly, we can exclude the case of even non-zero $c_i$. Using the previous part, we see that (\ref{eq23b:3coef}) can be rewritten as 
\[
\frac{q_0}{2}(1- b_jd_j) -\frac{1}{2}=\sum_{i=1}^N \frac{1}{2}a_ic_i.
\]
It implies that if there were more than one non-zero (odd) $c_i$, there would be an odd number of coefficients $c_i\neq 0$. That gives at least
\[
\frac{3}{4}+\frac{p}{4}+\frac{3q}{4}+\frac{pq}{4r_0^2}
\] 
on the right side of (\ref{eq23b:1coef}), which can be satisfied only for $q\leq 15$, which we exclude in the statement. Thus, there is only one non-zero $|c_k|=1$.

Now, we will discuss the case when $j=k$. Put $b_j=1$. Then:
\begin{enumerate}
\item If $d_k=1$, then (\ref{eq23b:2coef}) and (\ref{eq23b:4coef}) can be expressed as $2=a_j+c_jp_0$ and $2=a_j+c_jr_0$. However, these equalities can be both satisfied only if $0=c_j(p_0-r_0)$, which is impossible for $c_j\neq 0$ and $p_0\neq r_0$.
\item If $d_k=-1$, in a similar way, we get equalities $2=a_j-c_jp_0$ and $2=-a_j+c_jr_0$, which leads to $p_0=r_0+4$, $c_j=-1$ and $a_j=2-p_0$. When we insert this into (\ref{eq23b:3coef}), we obtain $p_0=2q_0+1$. Then the right side of (\ref{eq23b:1coef}) can be smaller than the left one only if $(p_0-2)^2\leq 33$. This leads to $p_0=7$, $q_0=3$ and $r_0=3$, which cannot be satisfied for square-free $p$, $q$ and $r$.    
\end{enumerate} 

Thus, $j\neq k$. If we put $c_k=-1$, we get $a_j=2$, $d_j=1$ and $a_k=1$. That gives the decomposition
\[
\alpha =7+ \Bigg(1+\frac{\sqrt{p}+\frac{\sqrt{pq}}{r_0}}{2}\Bigg)^2+\Bigg(\frac{1-\sqrt{q}}{2}\Bigg)^2,
\]
and for $7$, we need at least $4$ squares. By this, the proof for $q_0\geq 2$ is completed.

\bigskip

For $q_0=1$, likewise, we can use the element from the second part of Proposition \ref{prop:B1}, namely
\[
\beta=7+\Bigg(1+\frac{\sqrt{p}+\frac{\sqrt{pq}}{r_0}}{2}\Bigg)^2+\Bigg(\frac{\sqrt{p}-\frac{\sqrt{pq}}{r_0}}{2}\Bigg)^2=8+\frac{p}{2}+\frac{pq}{2r_0^2}+\sqrt{p}+\sqrt{r}.
\]      
\end{proof}

And finally, we will focus on bases of type (4).

\begin{prop} \label{prop:B4}
Let $K=\Q(\sqrt{p},\sqrt{q})$ where $p,q>1$ are square-free positive integers such that $p,q\equiv 1\;(\textup{mod }4)$ and $p,q,r\geq 17$. Then $\P(\O_K)\geq 6$. 
\end{prop}

\begin{proof}
Without loss of generality, we can assume $p<q<r$, i.e., $r_0<q_0<p_0$. Moreover, since the case when $r_0=1$ was already resolved in \cite{KRS}, we can assume $r_0\geq 3$.

For this basis, we will discuss the element
\begin{multline*}
\alpha=7+\Bigg(\frac{1+\sqrt{p}}{2}\Bigg)^2+\Bigg(\frac{1+\sqrt{p}+\sqrt{q}+\sqrt{r}}{4}\Bigg)^2\\=\frac{117}{16}+\frac{5p}{16}+\frac{q}{16}+\frac{r}{16}+\frac{p_0+5}{8}\sqrt{p}+\frac{q_0+1}{8}\sqrt{q}+\frac{r_0+1}{8}\sqrt{r}.
\end{multline*}
Therefore, as before, we get equalities
\begin{align}
\frac{117}{16}+\frac{5p}{16}+\frac{q}{16}+\frac{r}{16}&=\sum_{i=1}^N\Big(\frac{a_i^2}{16}+\frac{b_i^2p}{16}+\frac{c_i^2q}{16}+\frac{d_i^2pq}{16r_0^2}\Big),\label{eq4:1coef}\\
\frac{p_0+5}{8}&=\sum_{i=1}^N\Big(\frac{1}{8} a_ib_i+\frac{1}{8}c_id_ip_0\Big),\label{eq4:2coef}\\
\frac{q_0+1}{8}&=\sum_{i=1}^N\Big( \frac{1}{8}a_ic_i+\frac{1}{8}b_id_iq_0\Big),\label{eq4:3coef}\\
\frac{r_0+1}{8}&=\sum_{i=1}^N\Big(\frac{1}{8} b_ic_ir_0+\frac{1}{8}a_id_i\Big)\label{eq4:4coef}
\end{align}
where $a_i,b_i,c_i,d_i\in\Z$. Since $r_0\geq 3$, $q_0\geq 7$ and $p_0\geq 11$, there must be at least one non-zero $a_i$, $b_i$, $c_i$ and $d_i$. Moreover, we can easily check that every non-zero $d_i$ must satisfy $|d_i|=1$.

Now, let us assume that there exists non-zero even $c_i$, let say $c_k$. Under this assumption, we can derive that $|c_k|=2$, $b_k=d_k=0$, and there is only one index $j$ with $|b_j|=|c_j|=|d_j|=1$, and $b_i=c_i=d_i=0$ for $i\neq j,k$. Put $d_j=1$. Then we get the following two cases:
\begin{enumerate}
\item If $b_j=1$, then from (\ref{eq4:2coef}) and (\ref{eq4:4coef}), we obtain $p_0+5=a_j+c_jp_0$ and $r_0+1=a_j+c_jr_0$. These two equations are both satisfied only if $4=(c_j-1)(p_0-r_0)$, which is impossible for $p_0-r_0\geq 8$.
\item If $b_j=-1$, we similarly obtain $p_0+5=-a_j+c_jp_0$ and $r_0+1=a_j-c_jr_0$. This leads to $p_0+r_0+6=c_j(p_0-r_0)$. If $c_j=-1$, the right side is negative. If $c_j=1$, we can deduce that $2r_0+6=0$, which is impossible for $r_0\geq 3$.  
\end{enumerate}
Therefore, every non-zero coefficient $c_i$ is odd, and, moreover, we can easily check that $|c_i|\leq 1$ for all $i$.

Now, suppose that there exists $j\neq k$ such that $b_j$, $c_j$, $d_j$, $b_k$, $c_k$ and $d_k$ are all odd. In this case, easily, if $i\neq j,k$, then $b_i=c_i=d_i=0$, and $|b_j|=|c_j|=|d_j|=|b_k|=|c_k|=|d_k|=1$. Put $d_j=d_k=1$. Then we get:
\begin{enumerate}
\item If $b_j=b_k=1$, then the relations (\ref{eq4:2coef}) and (\ref{eq4:4coef}) give $p_0+5=a_j+a_k+(c_j+c_k)p_0$ and $r_0+1=a_j+a_k+(c_j+c_k)r_0$. This implies $4=(c_j+c_k-1)(p_0-r_0)$, which is impossible for $p_0-r_0\geq 8$. 
\item If $b_j=b_k=-1$, we similarly obtain $p_0+r_0+6=(p_0-r_0)(c_j+c_k)$. Clearly, $c_j+c_k\in\{-2,0,2\}$, but $c_j+c_k=-2,0$ can be immediately excluded as $p_0+r_0+6>0$. If $c_j+c_k=2$, we get $c_j=c_k=1$, $p_0=3r_0+6$ and $a_j+a_k=3r_0+1$. Having this, (\ref{eq4:3coef}) gives $q_0+1=a_j+a_k-2q_0=3r_0+1-2q_0$, i.e., $3q_0+1=3r_0+1$, which is impossible for $q_0\neq r_0$.
\item If $b_j=1$ and $b_k=-1$ (the case $b_j=-1$ and $b_k=1$ is analogous), we obtain $p_0+5=a_j-a_k+(c_j+c_k)p_0$ and $r_0+1=a_j+a_k+(c_j-c_k)r_0$. 
\begin{enumerate}
\item If $c_j=c_k=1$, then $a_j+a_k=r_0+1$ and (\ref{eq4:3coef}) gives $q_0+1=a_j+a_k+(1-1)q_0=r_0+1$, which is impossible for $q_0\neq r_0$.
\item If $c_j=1$ and $c_k=-1$, we obtain $p_0+5=a_j-a_k$, which leads to $q_0+1=a_j-a_k=p_0+5$. That is impossible for $q_0<p_0$.
\item If $c_j=-1$ and $c_k=1$, we again get $p_0+5=a_j-a_k$, which implies $q_0+1=-a_j+a_k=-p_0-5$, which is impossible.
\item If $c_j=c_k=-1$, then $r_0+1=a_j+a_k$ and $q_0+1=-a_j-a_k=-r_0-1$.  However, this is impossible for $q_0,r_0>0$.      
\end{enumerate} 
\end{enumerate}
Thus, we have exactly one triple of odd coefficients $b_j$, $c_j$ and $d_j$ with $|c_j|=|d_j|=1$, and $c_i=d_i=0$ for $i\neq j$.

Easily, $|b_j|\leq 3$. Set $d_j=1$. If $c_j=-1$, the equations (\ref{eq4:3coef}) and (\ref{eq4:4coef}) imply $q_0+r_0+2=b_j(q_0-r_0)$, which can be satisfied only for $b_j=1,3$. The case $b_j=1$ gives $2r_0+2=0$, which is impossible. On the other hand, $b_j=3$ leads to $a_j=4r_0+1$. Then the right side of (\ref{eq4:1coef}) is at least
\[
\frac{(4r_0+1)^2}{16}+\frac{9p}{16}+\frac{q}{16}+\frac{r}{16},
\]
and it is lower than the left side only if $(4r_0+1)^2+4p\leq 117$, which is impossible for $r_0\geq 3$. Thus, $c_j=1$. 

Since $d_j=c_j=1$, (\ref{eq4:3coef}) and (\ref{eq4:4coef}) gives $a_j=b_j=1$. Then (\ref{eq4:2coef}) leads to $p_0+5=\sum_{i\neq j}a_ib_i+1+p_0$, i.e., $4=\sum_{i\neq j}a_ib_i=4\sum_{i\neq j}a'_ib'_i$ where $a_i=2a'_i$ and $b_i=2b'_i$. Moreover, all non-zero $a'_i$ and $b'_i$ are odd, and there is an odd number of such non-zero pairs to get $1=\sum_{i\neq j} a'_ib'_i$. If there exist three such pairs, then the right side of (\ref{eq4:1coef}) is too large. Thus, there is only one index $k\neq j$ such that $b_k\neq 0$, and without loss of generality, we get $a_k=c_k=2$. Thus, the only possible relevant decomposition of $\alpha$ is of the form 
\[
\alpha=7+\Bigg(\frac{1+\sqrt{p}}{2}\Bigg)^2+\Bigg(\frac{1+\sqrt{p}+\sqrt{q}+\sqrt{r}}{4}\Bigg)^2,
\] 
and for $7$, we need at least $4$ squares. Therefore, $\P(\O_K)\geq 6$ even in this case.   
\end{proof}

\begin{proof}[Proof of Theorem \ref{thm:main}]
The proof of Theorem \ref{thm:main} is covered by Propositions \ref{prop:B1}, \ref{prop:B23} and \ref{prop:B4}. 
\end{proof}

In a similar fashion, we will prove Theorem \ref{thm:main7}.

\begin{proof}[Proof of Theorem \ref{thm:main7}]
Recall that in this case, $p=(2n-1)(2n+1)$, $q=(2n-1)(2n+3)$, $r=(2n+1)(2n+3)$ and $n\geq 6$. Moreover, $K(\sqrt{p},\sqrt{q})$ has an integral basis of type (3). To give our assertion, let us consider the element of the form
\begin{align*}
\alpha&=7+\left(\frac{1-\sqrt{q}}{2}\right)^2+\left(1+\frac{\sqrt{p}+\frac{\sqrt{pq}}{r_0}}{2}\right)^2\left(2+\frac{-\sqrt{p}+\frac{\sqrt{pq}}{r_0}}{2}\right)^2\\
&=\frac{49}{4}+\frac{p}{2}+\frac{q}{4}+\frac{pq}{2r_0}-\sqrt{p}-\frac{1}{2}\sqrt{q}+3\frac{\sqrt{pq}}{r_0}.
\end{align*}
Similarly as in the proof of Proposition \ref{prop:B23}, we consider decompositions of the form
\begin{align}
\frac{49}{4}+\frac{p}{2}+\frac{q}{4}+\frac{pq}{2r_0}&=\sum_{i=1}^N\Big(\frac{a_i^2}{4}+\frac{b_i^2p}{4}+\frac{c_i^2q}{4}+\frac{d_i^2pq}{4r_0^2}\Big),\label{eq7:1coef}\\
-1&=\sum_{i=1}^N\Big(\frac{1}{2} a_ib_i+\frac{1}{2}c_id_ip_0\Big),\label{eq7:2coef}\\
-\frac{1}{2}&=\sum_{i=1}^N\Big( \frac{1}{2}a_ic_i+\frac{1}{2}b_id_iq_0\Big),\label{eq7:3coef}\\
3&=\sum_{i=1}^N\Big(\frac{1}{2} b_ic_ir_0+\frac{1}{2}a_id_i\Big)\label{eq7:4coef}
\end{align}
where $a_i,b_i,c_i,d_i\in\Z$. Since $r_0\geq 11$, $q_0\geq 13$ and $p_0\geq 15$, we can again conclude the existence of at least one non-zero $a_i$, $b_i$, $c_i$ and $d_i$. 

If $|d_i|\geq 2$, then
\[
\frac{1}{4}+\frac{p}{4}+\frac{q}{4}+\frac{pq}{r_0^2}\leq \frac{49}{4}+\frac{p}{2}+\frac{q}{4}+\frac{pq}{2r_0^2},
\]
i.e., $r+(r-p)\leq 48$, which is not true for $r\geq 195$. Thus, $|d_i|\leq 1$ for all $i$. Using similar arguments and explicit expressions for $p$, $q$ and $r$, we can show $|b_i|\leq 1$ and $|c_i|\leq 1$ for all $i$. Moreover, in the same manner, there exist at most two non-zero coefficients $d_i$, and since $|b_i|\leq 1$, the same is also true for coefficients $b_i$. 

First of all, let us assume that there is only one pair of non-zero coefficients $b_i$ and $d_i$, let say $(b_j,d_j)\in\{(1,1),(-1,1)\}$ where, without loss of generality, we can suppose $d_j=1$. Then (\ref{eq7:2coef}) and (\ref{eq7:4coef}) give $-2=a_jb_j+c_jd_jp_0$ and $6=b_jc_jr_0+a_jd_j$. We obtain:
\begin{enumerate}
\item If $b_j=1$, then $8=c_j(r_0-p_0)=-4c_j$, which is a contradiction with $|c_j|\leq 1$.
\item If $b_j=-1$, then $4=c_j(p_0-r_0)=4c_j$, i.e., $c_j=1$ and $a_j=p_0+2=2n+5$. However, (\ref{eq7:3coef}) implies 
\[
-1=\sum_{i\neq j}a_ic_i+a_jc_j+b_jc_jq_0=\sum_{i\neq j}a_ic_i+p_0+2-q_0=\sum_{i\neq j}a_ic_i+4.
\]
Thus, there exists at least one $c_i\neq 0$ such that $i\neq j$. Then \ref{eq7:1coef} gives
\[
\frac{(2n+5)^2}{4}+\frac{p}{4}+\frac{q}{4}+\frac{r}{4}+\frac{1}{4}+\frac{q}{4}\leq \frac{49}{4}+\frac{p}{2}+\frac{q}{4}+\frac{r}{2},
\]
which can be rewritten as $(2n+5)^2\leq 48+r+p-q\leq 48+r$, which is never true for $n\geq 6$. 
\end{enumerate}

Therefore, we have exactly two pairs $(b_j,d_j),(b_k,d_k)\in\{(1,1),(-1,1)\}$ with $j\neq k$. In that case, there exists exactly one non-zero $c_i$, let say $c_l\in\{-1,1\}$. Moreover, (\ref{eq7:3coef}) implies $-1=a_lc_l+(b_jd_j+b_kd_k)q_0$. If $b_jd_j+b_kd_k=\pm 2$, then $|a_l|\geq 2q_0-1=4n+1$. However, for such coefficients $a_l$, the right side of (\ref{eq7:1coef}) is too large. Thus, $b_jd_j+b_kd_k=0$, and without loss of generality, $b_j=1$ and $b_k=-1$. Moreover, $a_lc_l=-1$.

If $l=k,l$, then, similarly as before, (\ref{eq7:2coef}) and (\ref{eq7:4coef}) lead to a contradiction for both cases of the pair $(a_l,c_l)\in\{(1,-1),(-1,1)\}$. Thus, $l\neq j,k$, and, without loss of generality, $a_j=1$ and $c_j=-1$. Furthermore, from (\ref{eq7:2coef}) and (\ref{eq7:4coef}), we also have $a_j=2$ and $a_k=4$. It follows that our decomposition must be of the form
\[
\alpha=7+\left(\frac{1-\sqrt{q}}{2}\right)^2+\left(1+\frac{\sqrt{p}+\frac{\sqrt{pq}}{r_0}}{2}\right)^2\left(2+\frac{-\sqrt{p}+\frac{\sqrt{pq}}{r_0}}{2}\right)^2.
\]
Since we need at least $4$ squares to express $7$, the proof is completed.         
\end{proof}

\section{Indecomposable integers in biquadratic fields} \label{sec:indebiqua}

Now, we will focus our attention on indecomposable integers in biquadratic fields. Recall that these algebraic integers are totally positive and cannot be expressed as a sum of two elements from $\O_K^{+}$. In the following, we will study three subfamilies of biquadratic fields from \cite{Man}. For them, we will find an upper bound on the norm of indecomposable integers, and we will also study the minimal traces of their indecomposables. For that, we need to know a form of the codifferent $\O_K^{\vee}$.

All of these families are of type (3). As indicated in Section \ref{sec:preli}, it can be easily checked that a $\Z$-basis of $\O_K^{\vee}$ for type (3) can consist of the elements
\begin{align*}
\varphi_1&=\frac{1}{4}+\frac{1}{4p_0r_0}-\frac{1}{2p_0r_0}\frac{1+\sqrt{q}}{2},\\
\varphi_2&=\left(\frac{1}{4p_0q_0}+\frac{1}{4q_0r_0}\right)\sqrt{p}-\frac{1}{2p_0q_0}\frac{\sqrt{p}+\sqrt{r}}{2},\\
\varphi_3&=-\frac{1}{2p_0r_0}+\frac{1}{p_0r_0}\frac{1+\sqrt{q}}{2},\\
\varphi_4&=-\frac{1}{2p_0q_0}\sqrt{p}+\frac{1}{p_0q_0}\frac{\sqrt{p}+\sqrt{r}}{2}.
\end{align*}

\subsection{First family}
The first family is the one from Theorem \ref{thm:main7}, which is defined by 
\[
p=(2n-1)(2n+1), \quad q=(2n-1)(2n+3), \quad r=(2n+1)(2n+3)
\]
where $n\geq 6$ is such that $p,q,r$ are all square-free. The fundamental units of the corresponding quadratic subfields are $\varepsilon_p=2n+\sqrt{p}$, $\varepsilon_q=\frac{2n+1+\sqrt{q}}{2}$ and $\varepsilon_r=2n+2+\sqrt{r}$, and they are all totally positive. Note that in these quadratic subfields, only totally positive units are indecomposable. On the other hand, the situation in $\Q(\sqrt{p},\sqrt{q})$ is much more complicated as proved by Man \cite{Man} in the following theorem:

\begin{theorem}[{\cite[Theorem 1.6]{Man}}]
let $p,q,r$ be as above. Then up to multiplication by totally positive units, the indecomposable integers in $\Q(\sqrt{p},\sqrt{q})$ are
$1$, $\frac{\varepsilon_p^{-1}+\varepsilon_r}{2}$, $\mu=n+\frac{3}{2}+\frac{1}{2}\sqrt{p}+\frac{1}{2}\sqrt{q}+\frac{1}{2}\sqrt{r}$,
\[
\alpha_t=1+\varepsilon_p+t(\mu-1), \quad \beta_t=\frac{1+\varepsilon_p\varepsilon_r}{2}+\varepsilon_p+t(\mu-1), \quad 3\leq t\leq 2n-2,
\]
\[
\gamma_t=1+\varepsilon_q^{-1}+t(\mu-1), \quad \delta_t=\frac{1+\varepsilon_p\varepsilon_r}{2}+\varepsilon_q^{-1}+t(\mu-1), \quad 4\leq t\leq 2n-1,
\]
\[
\omega_t=\frac{\varepsilon_r^{-1}+\varepsilon_p}{2}+t(\mu-2), \quad 2\leq t\leq 2n-1.
\]
\end{theorem} 
Moreover, we have the following relations among these indecomposable integers:
\begin{align*}
\delta_{2n-1-(t-3)}&=\sigma_3(\alpha_t)\varepsilon_r,\\
\beta_{2n-2-(t-3)}&=\sigma_3(\beta_t)\varepsilon_r,\\
\gamma_{t+1}&=\sigma_4(\beta_t)\varepsilon_p,\\
\omega_{2n-1-(t-2)}&=\sigma_2(\omega_t)\varepsilon_p.
\end{align*}
Therefore, regarding norms and minimal traces, it is enough to consider the elements $1$, $\frac{\varepsilon_p^{-1}+\varepsilon_r}{2}$, $\mu$, $\alpha_t$, $3\leq t\leq 2n-2$, $\beta_t$, $3\leq t\leq n$, and $\omega_t$, $2\leq t\leq n$.

Now, we will prove Theorem \ref{thm:main_norms}:

\begin{proof}[Proof of Theorem \ref{thm:main_norms}]
In biquadratic fields, which are Galois, it is not so hard to compute norms of elements. In particular, we have $N\left(\frac{\varepsilon_p^{-1}+\varepsilon_r}{2}\right)=(2n+1)^2$, $N(\mu)=4$,
\begin{align*}
N(\alpha_t)&=((4n+2)(t+1)-t^2)^2,\\
N(\beta_t)&=t^3-(4n+2)t^2-(4n^2+12n+3)t^2+(16n^3+40n^2+24n+4)t\\&\hspace{8cm}+16n^3+48n^2+44n+13,\\
N(\omega_t)&=(2n+1+(4n+2)t-2t^2)^2.
\end{align*}
Moreover, we can immediately exclude $\mu$ as its norm $4$ is too small in comparion with, e.g., $(2n+1)^2$ for $n\geq 6$.

We will start with elements $\alpha_t$. Let us set $g_1(t)=(4n+2)(t+1)-t^2$. Since derivative of $g_1(t)$, which is $4n+2-2t$, has zero in $t=2n+1$, the polynomial $g_1(t)$ is a monotonic function on $[3,2n-2]$. Therefore, the largest value among norms of $\alpha_t$ is attained for either $t=3$, or $t=2n-2$. In our case, the latter is true, and we obtain the bound $N(\alpha_{2n-2})=16 n^4+64n^3+16n^2-96n+36$.

We will proceed with $\beta_t$.
Let $g_2(t)=N(\beta_t)$. Then $g_2$ is a cubic polynomial in $t$ with a positive leading coefficient and $g_2(2n+3)=-16n^3-48n^2-20n+25<0$ for $n\geq 6$. Therefore, its maximum on our interval is attained in $t=n$ and equals $g_2(n)=9n^4+42n^3+69n+48n+13$.

Similarly, for $\omega_t$, we know that $N(\omega_{n})=N(\omega_{2n-1-(t-2)})$. Therefore, $2n+1+(4n+2)t-2t^2$ is monotonic on the interval  $[2,n]$. The largest norm is thus attained in some of the border points, and, finally, we get the bound $N(\omega_n)=4n^4+16n^3+20n^2+8n+1$.  

Thus, we are left with the norms $(2n+1)^2$, $16n^4+64n^3+16n^2-96n+36$, $9n^4+42n^3+69n+48n+13$ and $4n^4+16n^3+20n^2+8n+1$, from which $16n^4+64n^3+16n^2-96n+36$ is the largest for $n\geq 6$.     
\end{proof}

Now, we will discuss minimal traces of our elements.

\begin{prop}
We have:
\begin{enumerate}
\item $\textup{minTr}\left(\frac{\varepsilon_p^{-1}+\varepsilon_r}{2}\right)=\textup{minTr}\left(\mu\right)=1$,
\item $\textup{minTr}\left(\alpha_t\right)=\textup{minTr}\left(\beta_t\right)=2$ for all $3\leq t\leq 2n-2$,
\item $\textup{minTr}\left(\gamma_t\right)=\textup{minTr}\left(\delta_t\right)=2$ for all $4\leq t\leq 2n-1$,
\item $\textup{minTr}\left(\omega_t\right)=2$ for all $2\leq t\leq 2n-1$. 
\end{enumerate}
\end{prop}

\begin{proof}
Again, we can restrict to elements $\frac{\varepsilon_p^{-1}+\varepsilon_r}{2}$, $\mu$, $\alpha_t$, $3\leq t\leq 2n-2$, $\beta_t$, $3\leq t\leq n$, and $\omega_t$, $2\leq t\leq n$; the rest is a consequence of the above relations. 

In the first part, we see that $\frac{\varepsilon_p^{-1}+\varepsilon_r}{2}=2n+1-\sqrt{p}+\frac{\sqrt{p}+\sqrt{r}}{2}$. If we take the element of the codifferent $\varphi_1+(2n-1)\varphi_2-(n-1)\varphi_3-\varphi_3$, it gives trace $1$ with $\frac{\sqrt{p}+\sqrt{r}}{2}$. Thus, it suffices to prove that it is totally positive. It can be easily verified that $\varphi_1+(2n-1)\varphi_2-(n-1)\varphi_3-\varphi_3$ is a root of
\[
g(x)=x^4-x^3+\frac{4n+3}{8n^2+16n+6}x^2-\frac{1}{16n^2+32n+12}x+\frac{1}{16(4n^2+8n+3)^2},
\]
which gives our assertion.

We will proceed with the element $\mu=n+1+\frac{1+\sqrt{q}}{2}+\frac{\sqrt{p}+\sqrt{r}}{2}$. To get trace $1$, we can use the element of the codifferent $\varphi_1-(2n-1)\varphi_2+n\varphi_3-2n\varphi_4$, which is also a root of $g$ and, thus, totally positive. 

For $\alpha_t=(t+2)n+1+\sqrt{p}+t\frac{1+\sqrt{q}}{2}+t\frac{\sqrt{p}+\sqrt{r}}{2}$, we have 
\[
\text{Tr}(\alpha_t(\varphi_1-(2n-1)\varphi_2+n\varphi_3-2n\varphi_4))=2,
\]
which gives an upper bound on $\text{minTr}(\alpha_t)$. We will show that this is, in fact, $\text{minTr}(\alpha_t)$. Let us suppose that there is an element $a\varphi_1+b\varphi_2+c\varphi_3+d\varphi_4\in\O_K^{\vee,+}$ such that
\[
\text{Tr}(\alpha_t(a\varphi_1+b\varphi_2+c\varphi_3+d\varphi_4))=((t+2)n+1)a+b+tc+td=1.
\]
That gives $b=1-((t+2)n+1)a-t(c+d)$. If we multiply $a\varphi_1+b\varphi_2+c\varphi_3+d\varphi_4$ by totally positive elements $\alpha$, we need to get positive traces. Therefore, we will choose suitable elements $\nu$ to get restriction on $a$, $b$, $c$ and $d$. For $\nu=1$, we obtain the trace $a$, which implies $a>0$. Moreover,
\begin{align*}
\text{Tr}(\alpha_{t-1}(a\varphi_1+(1-((t+2)n+1)a-t(c+d))\varphi_2+c\varphi_3+d\varphi_4))&=1-an-c-d>0,\\
\text{Tr}(\alpha_{t+1}(a\varphi_1+(1-((t+2)n+1)a-t(c+d))\varphi_2+c\varphi_3+d\varphi_4))&=1+an+c+d>0.
\end{align*}
That gives $an+c+d=0$, i.e., our element of the codifferent is of the form
\[ 
a\varphi_1+(1-(2n+1)a)\varphi_2+c\varphi_3-(an+c)\varphi_4)
\]
for some $a,c\in\Z$. Note that both $\alpha_2$ and $\alpha_{2n-1}$ are totally positive; see \cite[Subsection 5.3]{Man}. Then
\[
\text{Tr}(\varepsilon_p(a\varphi_1+(1-(2n+1)a)\varphi_2+c\varphi_3-(an+c)\varphi_4))=1-a>0.
\]
However, that is a contradiction with $a>0$. Thus, we obtain $\text{minTr}(\alpha_t)=2$.  

The indecomposables 
\[
\beta_t=2n^2+3n+tn+2\sqrt{p}+(2n+1+t)\frac{1+\sqrt{q}}{2}+(2n+t)\frac{\sqrt{p}+\sqrt{r}}{2}
\]
also give trace $2$ with the element $\varphi_1-(2n-1)\varphi_2+n\varphi_3-2n\varphi_4$. Similarly, as before, if there were an element $\delta\in\O_K^{\vee,+}$ such that $\text{Tr}(\beta_t\delta)=1$, it would have to be of the form 
\[
a\varphi+\frac{1}{2}(1-(2n^2+3n+tn)a-(2n+t)(c+d)-c)\varphi+c\varphi_3+d\varphi_4
\]
for suitable $a,c,d\in\Z$ and $a>0$. Multiplying by $\beta_{t-1}$ and $\beta_{t+1}$, respectively, implies $an+c+d=0$, i.e., $\delta=a\varphi+\frac{1}{2}(1-3an-c)\varphi_2+c\varphi_3-(an+c)\varphi_4$. Then 
\begin{align*}
\text{Tr}(\sigma_3(\varepsilon_q)\delta)&=\text{Tr}\left(\Big(n+1-\frac{1+\sqrt{q}}{2}\Big)\delta\right)=(n+1)a-c>0,\\
\text{Tr}(\varepsilon_p\varepsilon_r\delta)&=1-(n+1)a+c>0
\end{align*}
However, the above inequalities cannot be both satisfied, giving $\text{minTr}(\beta_t)=2$.

Finally, for $\omega_t=2n+1+t(n-1)+\sqrt{p}+t\frac{1+\sqrt{q}}{2}+(t-1)\frac{\sqrt{p}+\sqrt{r}}{2}$, we will consider the element $\varphi_1-(2n-1)\varphi_2-(n-1)\varphi_3$. This element is a root of 
\[
x^4-x^3+\frac{3}{4n+6}x^2-\frac{1}{8n^2+16n+6}x+\frac{1}{4(4n^2+8n+3)^2};
\]
thus, it is totally positive and gives trace $2$ with $\omega_t$. Let us suppose that some $\delta=a\varepsilon_1+b\varepsilon_2+c\varepsilon_3+d\varepsilon_d\in\O_K^{\vee,+}$ satisfies $\text{Tr}(\omega_t\delta)=1$, which gives $b=1-(2n+1+t(n-1))a-(t-1)(c+d)-c$ and $a>0$. Then, again, by multiplying by $\omega_{t-1}$ and $\omega_{t+1}$, we obtain the condition $a(n-1)+c+d=0$, i.e., $\delta=a\varphi_1+(1-3an-c)\varphi_2+c\varphi_3-(a(n-1)+c)\varphi_4$. Let us mention that, $\omega_1$ and $\omega_{n+1}$ are both totally positive. Moreover,
\[
\text{Tr}(\varepsilon_p\varepsilon_q\delta)=\text{Tr}\left(\Big(2n^2+\sqrt{p}+2n\frac{1+\sqrt{q}}{2}+(2n-1)\frac{\sqrt{p}+\sqrt{r}}{2}\Big)\delta\right)=1-a>0,
\]
a contradiction. By this, we finished the proof. 
\end{proof}

\begin{proof}[Proof of Theorem \ref{thm:main_traces}]
As we have proved, every indecomposable integer in $K(\sqrt{p},\sqrt{q})$ has minimal trace $1$ or $2$. Note that for $1$, we can also use the element $\varphi_1-(2n-1)\varphi_2+n\varphi_3-2n\varphi_4$ to get $\text{minTr}(1)=1$. That gives our assertion.
\end{proof}

Moreover, we are able to provide a lower bound on the number of variables of universal diagonal quadratic forms.

\begin{prop}
Let $p,q,r$ be as above. Then every universal diagonal quadratic form over $O_K$ must have at least $\frac{2n-4}{3}$ variables. 
\end{prop}

\begin{proof}
We will apply Proposition \ref{prop:unilowvar}. If we consider $\delta=\varepsilon_1-(2n-1)\varepsilon_2+n\varepsilon_3-2n\varepsilon_4$, we know that $\text{Tr}(\alpha_t\delta)=\text{Tr}(\beta_t\delta)=2$ for all $3\leq t\leq 2n-2$, and, moreover, we can easily show that $\text{Tr}(\gamma_t\delta)=\text{Tr}(\delta_t\delta)=2$ for all $4\leq t\leq 2n-1$. That gives the bound $\frac{4(2n-4)}{12}=\frac{2n-4}{3}$.
\end{proof}

\subsection{Second family}

We will proceed with the family of the form $p=(2n-1)(2n+1)$, $q=(4n-3)(4n+1)$ and $r=pq$ where we assume $n\geq 9$ and that $p$ and $q$ are comprime and square-free. Moreover, similarly as before, $\varepsilon_p=2n+\sqrt{p}$, $\varepsilon_q=\frac{4n-1+\sqrt{q}}{2}$ and $\varepsilon_r=8n^2-2n-2+\sqrt{r}$. This family was also considered in \cite{Man} and has the following structure of indecomposable integers:

\begin{theorem}[{\cite[Theorem 5.3]{Man}}]
let $p,q,r$ be as above. Then up to multiplication by totally positive units, the indecomposable integers in $\Q(\sqrt{p},\sqrt{q})$ are
$1$, 
\[
\alpha_t=\frac{\varepsilon_p^{-1}+\varepsilon_r}{2}+t(\varepsilon_q-\varepsilon_p^{-1}), \quad 0\leq t\leq 2n-2,
\]
\[
\beta_t=\varepsilon_p^{-1}-\varepsilon_q+t(\varepsilon_p\varepsilon_q-1), \quad 1\leq t\leq 2n-1.
\]
\end{theorem}

As in the previous family, some of the indecomposables are associated with conjugates of other ones. In particular, we have $\beta_{2n-1-(t-1)}=\sigma_2(\alpha_t)\varepsilon_r$ for $1\leq t\leq 2n-2$. Thus, it is enough to discuss the elements $\alpha_t$ and $\beta_1$. The elements $\alpha_1$ and $\beta_1$ differ from the others in that they are associated with all their conjugates. 

For these biquadratic fields, we prove the following:

\begin{prop}
let $p,q,r$ be as above. If $\alpha\in\O_K^{+}$ is indecomposable in $\O_K$ for $K=\Q(\sqrt{p},\sqrt{q})$, then $N(\alpha)\leq 16n^4-8n^2+1$.
\end{prop}

\begin{proof}
Again, the norms of our elements are
\begin{align*}
N(\alpha_t)&= (t^2-4n^2+1)^2,\\
N(\beta_1)&=16n^2-16n+4.
\end{align*}
Let $h(t)=N(\alpha_t)$. It is easy to check that $h$ is a decreasing function between $0$ and $2n-2$. That gives the upper bound $N(\alpha_{0})=16n^4-8n^2+1>16n^2-16n+4$ for $n\geq 9$.
\end{proof}

Now, we will discuss the minimal traces.

\begin{prop}
Let $\alpha$ be an indecomposable integer in $\O_K$. Then $\textup{minTr}(\alpha)=1$.
\end{prop}

\begin{proof}
First, let us consider the elements $\alpha_t=4n^2-1-t+(t-1)\sqrt{p}+t\frac{1+\sqrt{q}}{2}+\frac{\sqrt{p}+\sqrt{r}}{2}$. The element of the codifferent $\varphi_1+(2n-1)\varphi_2-(2n-2)\varphi_3-(4n^2-2n-1)\varphi_4$ is a root of the polynomial
\begin{multline*}
x^4-x^3+\frac{8n^2-1}{32n^3+8n^2-8n-2}x^2-\frac{1}{4(16n^3+4n^2-4n-1)}x\\+\frac{1}{16(2n-1)^2(2n+1)^2(4n+1)^2}.
\end{multline*}
Therefore, it is totally positive and gives trace $1$ with $\alpha_t$.
Moreover, its trace with $1$ and $\beta_1=4n^2-2n+(2n-2)\sqrt{p}+(2n-1)\frac{1+\sqrt{q}}{2}+\frac{\sqrt{p}+\sqrt{r}}{2}$ is also $1$, which finishes the proof.
\end{proof}

\begin{prop}
Let $p,q,r$ be as above. Then every universal classical quadratic form over $O_K$ must have at least $\frac{2n+3}{2}$ variables. 
\end{prop}

\begin{proof}
We see that $\varphi_1+(2n-1)\varphi_2-(2n-2)\varphi_3-(4n^2-2n-1)\varphi_4\in\O_K^{\vee,+}$ produces trace $1$ with $\alpha_t$, $0\leq t\leq 2n-2$, and, moreover, with $1$ and $\beta_1$. Besides that, that is also true for the following elements:
\begin{align*}
\sigma_4(\beta_t)&=4n^2t-t+(1-2nt)\sqrt{p}+(1-2nt)\frac{1+\sqrt{q}}{2}+t\frac{\sqrt{p}+\sqrt{r}}{2},\\
\varepsilon_p^{-1}&=2n-\sqrt{p},\\
\varepsilon_q&=2n-1+\frac{1+\sqrt{q}}{2},\\
\varepsilon_r&=8n^2-2n-2-\sqrt{p}+2\frac{\sqrt{p}+\sqrt{r}}{2},\\
\sigma_4(\varepsilon_p\varepsilon_q)&=4n^2-2n\sqrt{p}-2n\frac{1+\sqrt{q}}{2}+\frac{\sqrt{p}+\sqrt{r}}{2},\\
\sigma_4(\varepsilon_p\varepsilon_r)&=16n^3-4n-1-(8n^2-2)\sqrt{p}-(8n^2-2)\frac{1+\sqrt{q}}{2}+4n\frac{\sqrt{p}+\sqrt{r}}{2},\\
\sigma_4(\varepsilon_q\varepsilon_r)&=16n^3-4n^2-4n-(8n^2-2n-2)\sqrt{p}-(8n^2-2n-2)\frac{1+\sqrt{q}}{2}\\
&\hspace{8cm}+(4n-1)\frac{\sqrt{p}+\sqrt{r}}{2}.
\end{align*} 
Applying Proposition \ref{prop:unilowvar}, we get the lower bound.  
\end{proof}

\subsection{Third family}

The last family from \cite{Man}, which we will examine, is of the form $p=(2n-1)(2n+1)$, $q=(4n-1)(4n+3)$ and $r=pq$ where $p$ and $q$ are assumed to be comprime and square-free, and $n\geq 2$. The corresponding fundamental units of quadratic subfields are $\varepsilon_p=2n+\sqrt{p}$, $\varepsilon_q=\frac{4n+1+\sqrt{q}}{2}$ and $\varepsilon_r=8n^2+2n-2+\sqrt{r}$. Note that these units are also totally positive. The structure of idecomposable integers is the following:

\begin{theorem}[{\cite[Theorem 5.4]{Man}}]
let $p,q,r$ be as above. Then up to multiplication by totally positive units, the indecomposable integers in $\Q(\sqrt{p},\sqrt{q})$ are
$1$, $\frac{\varepsilon_p^{-1}+\varepsilon_r}{2}$, 
\[
\alpha_t=\frac{\varepsilon_p+\varepsilon_r}{2}+t(\varepsilon_p-\varepsilon_q^{-1}), \quad 0\leq t\leq 2n-2,
\]
\[
\beta_t=\varepsilon_q^{-1}-\varepsilon_p+t(\varepsilon_p\varepsilon_q-1), \quad 1\leq t\leq 2n-1.
\]
\end{theorem}
Moreover, $\alpha_0=\sigma_3\left(\frac{\varepsilon_p^{-1}+\varepsilon_r}{2}\right)\varepsilon_p\varepsilon_r$ and $\beta_{2n-1-(t-1)}=\sigma_3(\alpha_t)\varepsilon_r$ for $1\leq t\leq 2n-2$. Thus, it suffices to examine elements $\alpha_t$ and $\beta_1$. Moreover, the situation is similar to the second family.

\begin{prop}
let $p,q,r$ be as above. If $\alpha\in\O_K^{+}$ is indecomposable in $\O_K$ for $K=\Q(\sqrt{p},\sqrt{q})$, then $N(\alpha)\leq 16n^4+16n^3-4n^2-4n+1$.
\end{prop}

\begin{proof}
In this case, we have
\begin{align*}
N(\alpha_t)&= (t^2+t -4n^2-2n+1)^2\\
N(\beta_1)&= 16n^2-8n+1.
\end{align*}
Moreover, the norm of $\alpha_t$ descreases between $0$ and $2n-2$. Thus, the largest norm is $N(\alpha_0)>N(\beta_1)$ for $n\geq 2$. 
\end{proof} 

As in the previous case, all indecomposables have minimal trace $1$.

\begin{prop}
Let $\alpha$ be an indecomposable integer in $\O_K$. Then $\textup{minTr}(\alpha)=1$.
\end{prop}

\begin{proof}
In this case, we consider the element of the codifferent $\varphi_1-(2n-1)\varphi_2+2n\varphi_3-(4n^2+2n-2)\varphi_4$, which is a root of the polynomial
\begin{multline*}
x^4-x^3+\frac{16 n^2+ 8 n-1}{64 n^3+ 64 n^2+ 4 n-6}x^2-\frac{1}{2 (32 n^3+ 32 n^2+ 2 n-3)}x\\+\frac{1}{4 (4 n-1)^2 (2 n+1)^2 (4 n+3)^2}.
\end{multline*}
This totally positive element produces trace $1$ both with integers $\alpha_t=4n^2+2n-1-t+t\sqrt{p}+t\frac{1+\sqrt{q}}{2}+\frac{\sqrt{p}+\sqrt{r}}{2}$ and $\beta_1=4n^2+(2n-1)\sqrt{p}+(2n-1)\frac{1+\sqrt{q}}{2}+\frac{\sqrt{p}+\sqrt{r}}{2}$, and, moreover, with $1$.
\end{proof}

This result also enables us to make the following conclusion.

\begin{prop}
Let $p,q,r$ be as above. Then every universal classical quadratic form over $\O_K$ must have at least $\frac{2n+3}{2}$ variables. 
\end{prop}

\begin{proof}
We will consider the element of the codifferent as above.
Besides $\alpha_t$, $\beta_1$ and $1$, we can use the following elements:
\begin{align*}
\sigma_4(\beta_t)&=4n^2t+2nt-t-(2nt+t-1)\sqrt{p}-(2nt-1)\frac{1+\sqrt{q}}{2}+t\frac{\sqrt{p}+\sqrt{r}}{2},\\
\varepsilon_p&=2n+\sqrt{p},\\
\varepsilon_q^{-1}&=2n+1-\frac{1+\sqrt{q}}{2},\\
\varepsilon_r&=8n^2+2n-2-\sqrt{p}+2\frac{\sqrt{p}+\sqrt{r}}{2},\\
\sigma_4(\varepsilon_p\varepsilon_q)&=4n^2+2n-(2n+1)\sqrt{p}-2n\frac{1+\sqrt{q}}{2}+\frac{\sqrt{p}+\sqrt{r}}{2},\\
\sigma_4(\varepsilon_p\varepsilon_r)&=16n^3+8n^2-4n-1-(8n^2+4n-2)\sqrt{p}-(8n^2-2)\frac{1+\sqrt{q}}{2}+4n\frac{\sqrt{p}+\sqrt{r}}{2},\\
\sigma_4(\varepsilon_q\varepsilon_r)&=16n^3+12n^2-2n-2-(8n^2+6n-1)\sqrt{p}-(8n^2+2n-2)\frac{1+\sqrt{q}}{2}\\
&\hspace{9cm}+(4n+1)\frac{\sqrt{p}+\sqrt{r}}{2}.
\end{align*}
Then the statement follows from Proposition \ref{prop:unilowvar}.
\end{proof}

\end{document}